\documentclass[leqno,11pt,a4paper]{amsart}
\usepackage{amsthm}
\usepackage{cite}
\usepackage{enumitem}
\usepackage{hyperref}
\usepackage{graphicx}
\usepackage{amssymb}
\usepackage{xcolor}
\hypersetup{
  colorlinks,
  linkcolor={blue!50!black},
  citecolor={blue!50!black},
  urlcolor={blue!80!black}
}
\usepackage{arydshln}
\usepackage{multirow}
\usepackage[all]{xy}
\usepackage{tikz}

\setlist[enumerate]{labelsep=*, leftmargin=1.5pc}
\setlist[enumerate]{label=\normalfont(\roman*), ref=\roman*}
\usepackage[margin=2.7cm]{geometry}
\graphicspath{{images/}}
\theoremstyle{plain}
\newtheorem{thm}{Theorem}[]
\newtheorem{pro}[thm]{Proposition}
\newtheorem{lem}[thm]{Lemma}

\theoremstyle{definition}
\newtheorem{dfn}[thm]{Definition}
\newtheorem{rem}[thm]{Remark}
\newtheorem{eg}[thm]{Example}
\newtheorem{algorithm}[thm]{Algorithm}
\DeclareMathOperator{\Conv}{Conv}
\DeclareMathOperator{\coeff}{coeff}

\DeclareMathOperator{\Hom}{Hom}

\DeclareMathOperator{\Pic}{Pic}

\DeclareMathOperator{\Newt}{Newt}

\newcommand{\cA}{\mathcal{A}}
\newcommand{\cO}{\mathcal{O}}
\newcommand{\QQ}{{\mathbb{Q}}}

\newcommand{\PP}{{\mathbb{P}}}
\newcommand{\XX}{{\mathbb{X}}}
\newcommand{\ZZ}{{\mathbb{Z}}}
\renewcommand{\AA}{{\mathbb{A}}}

\newcommand{\CC}{\mathbb{C}}

\newcommand{\LL}{\mathbb{L}}

\newcommand{\cM}{\mathcal{M}}
\newcommand{\Cstar}{\CC^\times}

\renewcommand{\emptyset}{\varnothing}
\begin{document}
\author[Coates]{Tom Coates}
\author[Kasprzyk]{Alexander Kasprzyk}
\author[Prince]{Thomas Prince}
\address{Department of Mathematics\\Imperial College London\\London, SW$7$\ $2$AZ\\UK}
\email{t.coates@imperial.ac.uk}
\email{a.m.kasprzyk@imperial.ac.uk}
\email{t.prince12@imperial.ac.uk}
\title{Laurent Inversion}
\maketitle
\begin{abstract}
There are well-understood methods, going back to Givental and Hori--Vafa, that to a Fano toric complete intersection~$X$ associate a Laurent polynomial~$f$ that corresponds to $X$ under mirror symmetry. We describe a technique for inverting this process, constructing the toric complete intersection~$X$ directly from its Laurent polynomial mirror~$f$. We use this technique to construct a new four-dimensional Fano manifold.
\end{abstract}
\section{Introduction}
Fano manifolds are basic building blocks in algebraic geometry, and the classification of Fano manifolds is a long-standing and important open problem. The classification in dimensions one and two has been known since the 19th century: there is a unique one-dimensional Fano manifold, the complex projective line, and there are ten deformation families of two-dimensional Fano manifolds, the del~Pezzo surfaces. The three-dimensional classification was completed in the 1990s by Mori and Mukai, building on the rank-1 classification by Fano in the 1930s and Iskovskikh in the 1970s~\cite{Iskovskih:1,Iskovskih:2,Iskovskih:anticanonical,Mori--Mukai:Manuscripta,Mori--Mukai:Tokyo,Mori--Mukai:Kinosaki,Mori--Mukai:Erratum,Mori--Mukai:Turin}. Very little is known about the classification of Fano manifolds in higher dimensions. 

In~\cite{CCGGK}, Coates--Corti--Galkin--Golyshev--Kasprzyk announced a program to find and classify Fano manifolds using Mirror Symmetry. These methods should work in all dimensions. Extensive computational experiments suggest that, under mirror symmetry, $n$-dimensional Fano manifolds correspond to certain Laurent polynomials in $n$ variables with very special properties. It is now understood how to recover the known classifications in low dimensions from this perspective~\cite{A+,ACGK,CCGK}, but in order to use this to gain new insight into Fano classification we need to solve two problems:
\begin{enumerate}[label=(\Alph*),ref=\Alph*]
\item what is the class of Laurent polynomials that correspond, under Mirror Symmetry, to Fano manifolds? \label{problem:A}
\item given such a Laurent polynomial $f$, how can we construct the corresponding Fano manifold $X$?  \label{problem:B}
\end{enumerate}
We believe that problem~\ref{problem:A} here is now solved. Fano manifolds conjecturally correspond to \emph{rigid maximally mutable Laurent polynomials}. The correspondence here is (conjecturally) one-to-one, where we consider Fano manifolds up to deformation and Laurent polynomials up to certain birational changes of variable called\footnote{Mutations are close analogs of cluster transformations~\cite{FZ,GHKK} and wall-crossing formulae~\cite{Aur:1,Aur:2,KS}.} mutations~\cite{ACGK}. 
Maximally mutable Laurent polynomials~\cite{A+,Kasprzyk--Tveiten} are Laurent polynomials $f$ which admit, in a precise sense, as many mutations as possible; this notion makes sense in all dimensions. Maximally mutable Laurent polynomials typically occur in parametrised families, and those that do not are referred to as rigid. 

In this paper we make significant progress on problem~\ref{problem:B}. There are well-understood methods, going back to Givental and Hori--Vafa, that to a Fano toric complete intersection~$X$ associate a Laurent polynomial~$f$ that corresponds to $X$ under Mirror Symmetry. We describe a technique, \emph{Laurent inversion}, for inverting this process, constructing the toric complete intersection~$X$ directly from its Laurent polynomial mirror~$f$. In many cases this allows, given a Laurent polynomial $f$, the direct construction of a Fano manifold $X$ that corresponds to $f$ under Mirror Symmetry. Thus, in many cases, Laurent inversion solves problem~2 above. As proof of concept, in~\S\ref{sec:new_4d} below we construct a new four-dimensional Fano manifold, by applying Laurent inversion to a rigid maximally-mutable Laurent polynomial in four variables.

It is expected that, if a Fano manifold $X$ is mirror to a Laurent polynomial~$f$, then there is a degeneration from $X$ to the (singular) toric variety $X_f$ defined by the spanning fan of the Newton polytope of $f$. Thus one might hope to recover the Fano manifold $X$ from $f$ by smoothing $X_f$, for instance using the Gross--Siebert program~\cite{Gross--Siebert}. This works in dimension two~\cite{Prince}, but the higher-dimensional case is  significantly more involved. As we will see in~\S\ref{sec:embedded_degeneration} below, in many cases Laurent inversion constructs, along with $X$, an embedded degeneration from $X$ to the singular toric variety $X_f$ -- thus implementing the smoothing of $X_f$ expected from the Gross--Siebert program. Laurent inversion should therefore give a substantial hint as to the generalisations required to get a Gross--Siebert-style smoothing procedure working, in this context, in higher dimensions.

\section{Laurent Polynomial Mirrors for Toric Complete Intersections}
\label{sec:forward}

We begin by recalling how to associate to a toric complete intersection $X$ a Laurent polynomial that corresponds to $X$ under Mirror Symmetry. This question has been considered by many authors~\cite{Givental:toric,Hori--Vafa,Prz:1,Prz:2,Doran--Harder,CKP}, and we will give a construction which generalises and unifies all these perspectives below (in~\S\ref{sec:torus_charts}). Consider first the ambient toric variety or toric stack $Y$. We consider the case where:
\begin{equation}
  \label{eq:ambient_conditions}
  \begin{minipage}{0.94\linewidth}
    \begin{enumerate}
    \item $Y$ is a proper toric Deligne--Mumford stack;  
    \item the coarse moduli space of $Y$ is projective;  
    \item the generic isotropy group of $Y$ is trivial, that is, $Y$ is a toric \emph{orbifold}; and  
    \item at least one torus-fixed point in $Y$ is smooth. 
    \end{enumerate}
  \end{minipage}
\end{equation}
Conditions (i--iii) here are essential; condition~(iv) is less important and will be removed in~\S\ref{sec:torus_charts}.
In the original work by Borisov--Chen--Smith~\cite{Borisov--Chen--Smith}, toric Deligne--Mumford stacks are defined in terms of stacky fans. In our context, since the generic isotropy is trivial, giving a stacky fan that defines $Y$ amounts to giving a triple $(N;\Sigma;\rho_1,\ldots,\rho_R)$ where $N$ is a lattice, $\Sigma$ is a rational simplicial fan in $N \otimes \QQ$, and $\rho_1,\ldots,\rho_R$ are elements of $N$ that generate the rays of $\Sigma$. It will be more convenient for our purposes, however, to represent $Y$ as a GIT quotient $\big[ \CC^R/\!\!/_{\!\omega} (\Cstar)^r\big]$. Any such $Y$ can be realised this way, as we now explain.


\begin{dfn} \label{dfn:GIT_data}
  We say that $(K;\LL;D_1,\ldots,D_R;\omega)$ are \emph{GIT data} if $K \cong (\Cstar)^r$ is a connected torus of rank~$r$; $\LL = \Hom(\Cstar,K)$ is the lattice of subgroups of $K$; $D_1,\ldots,D_R \in \LL^*$ are characters of $K$ that span a strictly convex full-dimensional cone in $\LL^* \otimes \QQ$, and $\omega \in \LL^* \otimes \QQ$ lies in this cone.
\end{dfn}

GIT data $(K;\LL;D_1,\ldots,D_R;\omega)$ determine a quotient stack $\big[ V_\omega/K \big]$ with $V_\omega \subset \CC^R$, as follows. The characters $D_1,\ldots,D_R$ define an action of $K$ on $\CC^R$. Write $[R] := \{1,2,\ldots,R\}$. Say that a subset $I \subset [R]$ \emph{covers} $\omega$ if and only if $\omega = \sum_{i \in I} a_i D_i$ for some strictly positive rational numbers $a_i$, set $\cA_\omega = \{ I \subset [R]\mid \text{$I$ covers $\omega$}\}$, and set
\begin{align*}
  V_\omega = \bigcup_{I \in \cA_\omega} (\Cstar)^I \times \CC^{\bar{I}} && \text{where} && (\Cstar)^I \times \CC^{\bar{I}} = \big\{(x_1,\ldots,x_R) \in \CC^R\mid \text{$x_i \ne 0$ if $i \in I$}\big\}. 
\end{align*}
The subset $V_\omega \subset \CC^R$ is $K$-invariant, and $\big[ V_\omega/K\big]$ is the GIT quotient (stack) given by the action of $K$ on $\CC^R$ and the stability condition $\omega$. The  convexity hypothesis in Definition~\ref{dfn:GIT_data} ensures that $\big[ V_\omega/K\big]$ is proper.

\begin{rem}
  The quotient $\big[ V_\omega/K\big]$ here depends on $\omega$ only via the minimal cone $\sigma$ of the secondary fan such that $\omega \in \sigma$. The \emph{secondary fan} for GIT data $(K;\LL;D_1,\ldots,D_R;\omega)$ is the fan defined by the wall-and-chamber decomposition of the cone in $\LL^* \otimes \QQ$ spanned by $D_1,\ldots,D_R$, where the walls are given by the cones spanned by $\{D_i\mid i \in I\}$ such that $I \subset [R]$ and $|I|=r-1$. 
\end{rem}

\begin{dfn}
  \emph{Orbifold} GIT data are those such that the quotient $\big[ V_\omega/K\big]$ is a toric orbifold.
\end{dfn}

The quotient $\big[ V_\omega/K\big]$ is a toric Deligne--Mumford stack if and only if $\omega$ lies in the strict interior of a maximal cone in the secondary fan. A toric orbifold $Y$ satisfying the conditions~\eqref{eq:ambient_conditions} above arises as the quotient $\big[ V_\omega/K\big]$ for GIT data $(K;\LL;D_1,\ldots,D_R;\omega)$ as follows. Suppose that $Y$ is defined, as discussed above, by the stacky fan data $(N;\Sigma;\rho_1,\ldots,\rho_R)$. There is an exact sequence
\begin{equation}
  \label{eq:fan_sequence}
  \xymatrix{
    0 \ar[r] & \LL \ar[r] & \ZZ^R \ar[r]^\rho  & N \ar[r] & 0
  }
\end{equation}
where $\rho$ maps the $i$th element of the standard basis for $\ZZ^R$ to $\rho_i$; this defines $\LL$ and $K = \LL \otimes \Cstar$. Dualizing gives 
\begin{equation}
  \label{eq:divisor_sequence}
  \xymatrix{
    0 & \LL^* \ar[l] & (\ZZ^*)^R \ar[l]_D  & M \ar[l] & 0 \ar[l]
  }
\end{equation}
where $M := \Hom(N,\ZZ)$, and we set $D_i \in \LL^*$ to be the image under $D$ of the $i$th standard basis element for $(\ZZ^*)^R$. The stability condition $\omega$ is taken to lie in the strict interior of 
\[
C = \bigcap_{\text{maximal cones $\sigma$ of $\Sigma$}} C_\sigma
\]
where $C_\sigma$ is the cone in $\LL^* \otimes \QQ$ spanned by $\{D_i\mid i \in \sigma\}$; projectivity of the coarse moduli space of $Y$ implies that $C$ is a maximal cone of the secondary fan, and in particular that $C$ has non-empty interior.

We can reverse this construction, defining a stacky fan $(N;\Sigma;\rho_1,\ldots,\rho_n)$ from GIT data $(K;\LL;D_1,\ldots,D_R;\omega)$ such that $D_1,\ldots,D_R$ span $\LL^*$, as follows. The lattice $\LL$ and elements $D_1,\ldots,D_R \in \LL^*$ define the exact sequence~\eqref{eq:divisor_sequence}, and dualising gives~\eqref{eq:fan_sequence}. This defines the lattice $N$ and $\rho_1,\ldots,\rho_R$. The fan $\Sigma$ consists of the cones spanned by $\{\rho_i\mid i \in I\}$ where $I \subset [R]$ satisfies $[R]\setminus I \in \cA_\omega$.

\begin{rem}
  Once $K$,~$\LL$, and~$D_1,\ldots,D_R$ have been fixed, choosing $\omega$ such that the GIT data $(K;\LL;D_1,\ldots,D_R;\omega)$ define a toric Deligne--Mumford stack amounts to choosing a maximal cone in the secondary fan. 
\end{rem}

Under our hypotheses there is a canonical isomorphism between $\LL^*$ and the Picard lattice $\Pic(Y)$. We will denote the line bundle on $Y$ corresponding to a character $\chi \in \LL^*$ also by $\chi$.

\begin{dfn} \label{dfn:convex_partition_with_basis}
  Let $\Theta = (K;\LL;D_1,\ldots,D_R;\omega)$ be orbifold GIT data, and let $Y$ denote the corresponding toric orbifold. A \emph{convex partition with basis} for $\Theta$ is a partition $B,S_1,\ldots,S_k,U$ of $[R]$
  such that:
  \begin{enumerate}
  \item $\{D_b\mid b \in B\}$ is a basis for $\LL^*$; 
  \item $\omega$ is a non-negative linear combination of $\{D_b\mid b \in B\}$;
  \item each $S_i$ is non-empty; 
  \item for each $i \in [k]$, the line bundle $L_i := \sum_{j \in S_i} D_j$ on $Y$ is convex\footnote{A line bundle $L$ on a Deligne--Mumford stack $Y$ is convex if and only if $L$ is nef and is the pullback of a line bundle on the coarse moduli space $|Y|$ of $Y$ along the structure map $Y \to |Y|$. See~\cite{CGIJJM}.}; and 
  \item for each $i \in [k]$, $L_i$ is a non-negative linear combination of $\{D_b\mid b \in B\}$. 
  \end{enumerate}
  We allow $k=0$, and we allow $U=\emptyset$. 
\end{dfn}


\begin{rem}
  Since $\omega$ here is taken to lie in the strict interior of a maximal cone in the secondary fan, it is in fact a positive linear combination of $\{D_b\mid b \in B\}$. This positivity guarantees that the maximal cone spanned by $\{\rho_i\mid i \in [R]\setminus B\}$ defines a smooth torus-fixed point in $Y$.
\end{rem}

\begin{rem}
  It would be more natural to replace the condition that $L_i$ be convex here with the weaker condition that $L_i$ be nef. But, since we currently lack a Mirror Theorem that applies to toric complete intersections beyond the convex case, we will require convexity. If the ambient space $Y$ is a manifold, rather than an orbifold, then convexity and nef-ness coincide.
\end{rem}

Given:
\begin{equation}
  \label{eq:lots_of_data}
  \begin{minipage}{0.94\linewidth}
    \begin{enumerate}
    \item orbifold GIT data $\Theta = (K;\LL;D_1,\ldots,D_R;\omega)$;
    \item a convex partition with basis $B, S_1, \ldots, S_k, U$ for $\Theta$; and
    \item a choice of elements $s_i \in S_i$ for each $i \in [k]$;
    \end{enumerate}
  \end{minipage}
\end{equation}
we define a Laurent polynomial $f$, as follows. Without loss of generality we may assume that $B = [r]$. Writing $D_1,\ldots,D_R$ in terms of the basis $\{D_b\mid b \in B\}$ for $\LL^*$ yields an $r \times R$ matrix $\cM = (m_{i,j})$ of the form
\begin{equation}
  \label{eq:weight_matrix}
  \cM = \left(
    \begin{array}{c:ccc}
      \multirow{3}{*}{$\quad I_r\quad$} & m_{1,r+1} & \cdots & m_{1,R} \\
      & \vdots & & \vdots \\
      & m_{r,r+1} & \cdots & m_{r,R} \\
    \end{array}
  \right)
\end{equation}
where $I_r$ is an $r \times r$ identity matrix. Consider the function
\[
W = x_1 + x_2 + \cdots + x_R - k
\]
subject to the constraints
\begin{align}
  \label{eq:ambient_constraints}
  \prod_{j=1}^R x_j^{m_{i,j}} = 1 && i \in [r] \\
  \intertext{and}
  \label{eq:hypersurface_constraints}
  \sum_{j \in S_i} x_j = 1 && i \in [k]
\end{align}
For each $i \in [k]$, introduce new variables $y_j$, where $j \in S_i \setminus \{s_i\}$, and set $y_{s_i} = 1$. Solve the constraints~\eqref{eq:hypersurface_constraints} by setting:
\begin{align*}
  x_j = \frac{y_j}{\sum_{l \in S_i} y_l} && j \in S_i
\end{align*}
and express the variables $x_b$, $b \in B$, in terms of the $y_j$s and remaining $x_i$s using~\eqref{eq:ambient_constraints}. The function $W$ thus becomes a Laurent polynomial~$f$ in variables
\begin{equation}
  \label{eq:which_variables}
  \begin{aligned}
    x_i,\quad&\text{ where } i \in U,\\
    \text{ and }\quad
    y_j,\quad&\text{ where } j \in (S_1 \cup \cdots \cup S_k) \setminus \{s_1,\ldots,s_k\}.
  \end{aligned}
\end{equation}
We call the $x_i$ here the \emph{uneliminated variables}. 

Given data as in~\eqref{eq:lots_of_data}, let $f$ be the Laurent polynomial just defined. Let $Y$ denote the toric orbifold determined by $\Theta$, let $L_1,\ldots,L_k$ denote the line bundles on $Y$ from Definition~\ref{dfn:convex_partition_with_basis}, and let $X \subset Y$ be a complete intersection defined by a regular section of the vector bundle $\oplus_i L_i$. If $X$ is Fano, then Mirror Theorems due to Givental~\cite{Givental:toric}, Hori--Vafa~\cite{Hori--Vafa}, and Coates--Corti--Iritani--Tseng~\cite{CCIT:1,CCIT:2} imply that $f$ corresponds to $X$ under Mirror Symmetry  (c.f.~\cite[\S5]{CKP}). We say that $f$ \emph{is a Laurent polynomial mirror} for $X$.

\begin{rem} \label{rem:translation}
  If $f$ is a Laurent polynomial mirror for $X$ then the Picard--Fuchs local system for $f \colon (\Cstar)^n \to \CC$ coincides, after translation of the base if necessary, with the Fourier--Laplace transform of the quantum local system for~$X$; see~\cite{CCGGK,CCGK}. Thus we regard $f$ and $g := f-c$, where $c$ is a constant, as Laurent polynomial mirrors for the same manifold $Y$, since the Picard--Fuchs local systems for $f$ and $g$ differ only by a translation of the base (by~$c$).
\end{rem}

\begin{rem} \label{rem:reparametrization}
  If $f$ and $g$ are Laurent polynomials that differ by an invertible monomial change of variables then the Picard--Fuchs local systems for~$f$ and~$g$ coincide. Thus $f$ is a Laurent polynomial mirror for $X$ if and only if $g$ is a Laurent polynomial mirror for $X$.
\end{rem}

\begin{eg} \label{eg:cubic_forwards}
  Let $X$ be a smooth cubic surface. The ambient toric variety $Y = \PP^3$ is a GIT quotient $\CC^4 /\!\!/ \Cstar$ where $\Cstar$ acts on $\CC^4$ with weights $(1,1,1,1)$. Thus $Y$ is given by GIT data $(K;\LL;D_1,\ldots,D_4;\omega)$ with $K=\Cstar$, $\LL = \ZZ$, $D_1=D_2=D_3=D_4=1$, and $\omega=1$. We consider the convex partition with basis $B$,~$S_1$,~$\emptyset$, where $B = \{1\}$ and $S_1 = \{2,3,4\}$, and take $s_1 = 4$. This yields
  \[
  \cM = \left( 
    \begin{array}{c:ccc}
      1&1&1&1
    \end{array}
  \right)
  \]
  and
  \[
  W = x_1 + x_2 + x_3 + x_4 - 1
  \]
  subject to 
  \begin{align*}
    x_1 x_2 x_3 x_4 = 1 && \text{and} && x_2 + x_3 + x_4 = 1.
  \end{align*}
  We set:
  \begin{align*}
    x_1 = \frac{1}{x_2 x_3 x_4} &&
    x_2 = \frac{x}{1+x+y} &&                                  
    x_3 = \frac{y}{1+x+y} &&                                  
    x_4 = \frac{1}{1+x+y} &&                                  
  \end{align*}
  where, in the notation above, $x = y_2$ and $y = y_3$. Thus
  \[
  f =  \frac{(1+x+y)^3}{xy}
  \]
  is a Laurent polynomial mirror to $Y$.
\end{eg}

\begin{eg}
  Let $Y$ be the projective bundle $\PP\big(\cO \oplus \cO \oplus \cO(-1)\big) \to \PP^3$. This arises from the GIT data $(K;\LL;D_1,\ldots,D_7;\omega)$ where $K = (\Cstar)^2$, $\LL = \ZZ^2$, 
  \begin{align*}
    D_1 = D_4 = D_6 = D_7 = (1,0) && D_2 = D_3 = (0,1) && D_5 = (-1,1)
  \end{align*}
  and $\omega = (1,1)$. We consider the convex partition with basis $B,S_1,S_2,U$ where $B = \{1,2\}$, $S_1 = \{3,4\}$, $S_2 = \{5,6\}$, $U = \{7\}$. This yields:
  \[
  \cM =
  \left(
  \begin{array}{cc:ccccc}
    1&0&0&1&-1&1&1\\
    0&1&1&0&1&0&0
  \end{array}
  \right)
  \]
  Choosing $s_1 = 3$ and $s_2 = 5$, we find that
  \[
  f = \frac{(1+x)}{xyz} + (1+x)(1+y) + z
  \]
  Here, in the notation above, $x = y_4$, $y = y_6$, and $z = x_7$.
\end{eg}

\section{Laurent Inversion}

To invert the process described in~\S\ref{sec:forward}, that is, to pass from a Laurent polynomial $f$ to orbifold GIT data $\Theta$, a convex partition  with basis  $B,S_1,\ldots,S_k,U$ for $\Theta$, and elements $s_i \in S_i$, $i \in [k]$, would amount to expressing $f$ in the form 
\begin{equation}
  \label{eq:scaffolding}
  f =  f_1 + \cdots + f_r + \sum_{u \in U} x_u
\end{equation}
where
\[
f_a = \prod_{i=1}^k \prod_{j \in S_i} \left( \frac{\sum_{l \in S_i} y_l}{y_j} \right)^{m_{a,j}} \times \prod_{u \in U} x_u^{-m_{a,u}}.
\]
In favourable circumstances, we can obtain from a decomposition~\eqref{eq:scaffolding} a smooth toric orbifold $Y$ and convex line bundles $L_1,\ldots,L_k$ on $Y$ such that the complete intersection $X \subset Y$ defined by a regular section of the vector bundle $\oplus_i L_i$ is Fano and corresponds to $f$ under Mirror Symmetry. In general there are many such decompositions of $f$. Not every decomposition gives rise to a smooth toric orbifold $Y$, for example because not every decomposition gives rise to valid GIT data\footnote{The characters $D_1,\ldots,D_R$ of $K = (\Cstar)^r$ defined, via equation~\eqref{eq:weight_matrix}, by a decomposition~\eqref{eq:scaffolding} may not span a strictly convex full-dimensional cone.}. Even when the decomposition~\eqref{eq:scaffolding} gives orbifold GIT data $(K;\LL;D_1,\ldots,D_R;\omega)$, and hence an ambient toric orbifold $Y$, it is not always possible to choose the stability condition $\omega$ such that $Y$ has a smooth torus-fixed point, or such that the line bundles $L_1,\ldots,L_k$ are simultaneously convex, or such that $X$ is Fano. In practice, however, this technique is surprisingly effective. 

\begin{dfn}
  We refer to a decomposition~\eqref{eq:scaffolding} as a \emph{scaffolding} for $f$, and to the Laurent polynomials $f_a$ involved as \emph{struts}.
\end{dfn}

\begin{algorithm}
  We remark -- and this is a key methodological point --  that scaffoldings of $f$ can be enumerated algorithmically. Let $A = \ZZ^s$ denote the lattice containing $\Newt{f}$. A partition $S'_1,\ldots,S'_k,U'$ of the standard basis for $A$, where we allow $k=0$ and allow $U' = \emptyset$, defines a collection of standard simplices
  \begin{align*}
    \Delta(i) = \Conv \big(\{0\} \cup S'_i\big) && i \in [k].
  \end{align*}
  We call a polytope $\Delta$ a \emph{strut} if it is a translation of a Minkowski sum of dilations of these standard simplices. A scaffolding~\eqref{eq:scaffolding} for $f$ determines a collection of struts $\Delta_a$ and lattice points $p_u$, each contained in $P := \Newt{f}$, where $\Delta_a = \Newt{f_a}$ and $p_u$ is the standard basis element corresponding to the uneliminated variable $x_u$. The struts $\Delta_a$ may overlap, and may overlap with the $p_u$. We refer to a collection $\{\Delta_a\mid a \in [r]\}$, $\{p_u\mid u \in U'\}$ of:
  \begin{enumerate}
  \item struts $\{\Delta_a\mid a \in [r]\}$ with respect to some partition $S_1',\ldots,S_k',U'$; and
  \item standard basis elements $\{p_u\mid u \in U'\}$;
  \end{enumerate}
  all of which are contained in a polytope $P$, as a \emph{scaffolding} for $P$.
  One can check whether a scaffolding for $\Newt{f}$ arises from a scaffolding~\eqref{eq:scaffolding} for $f$ by checking if the coefficients from the associated struts $f_a$ and uneliminated variables $x_u$ sum to give the coefficients of $f$. Since all coefficients of the struts $f_a$ are positive, only finitely many scaffoldings for $\Newt{f}$ need to be checked. We are free to relax our notion of scaffolding, demanding that the left- and right-hand sides of~\eqref{eq:scaffolding} agree only up to a constant monomial  -- see Remark~\ref{rem:translation}. This extra flexibility is often useful.
\end{algorithm}

\begin{rem}
  It is more meaningful, in view of Remark~\ref{rem:reparametrization}, to allow scaffoldings of $\Newt{f}$ that are based on a partition $S_1',\ldots,S_k',U'$ of an arbitrary basis for $A$, rather than the standard basis. For fixed $f$, only finitely many such generalised scaffoldings need be checked.
\end{rem}

\begin{eg}[$dP_3$]
Consider now the Laurent polynomial
\[
f=\frac{(1+x+y)^3}{xy}
\]
from Example~\ref{eg:cubic_forwards}. A scaffolding for $\Newt{f}$ is given by a single standard 2-simplex, dilated by a factor of three:
\begin{center}
\includegraphics{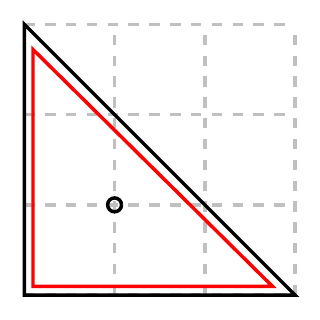}
\end{center}
Indeed $f$ is equal to a single strut, with no uneliminated variables. From this we read off $r=1$, $k = 1$, $B = \{1\}$, $S_1 = \{2,3,4\}$, $U = \varnothing$, and the exponents of the strut give:
\[
\cM = \left( 
  \begin{array}{c:ccc}
    1&1&1&1
  \end{array}
\right)
\]
This gives GIT data $\Theta = (K;\LL;D_1,\ldots,D_4;\omega)$ with $K=\Cstar$, $\LL = \ZZ$, $D_1=D_2=D_3=D_4=1$, and $\omega=1$; note that the secondary fan here has a unique maximal cone.
The corresponding toric variety is $Y = \PP^3$. The line bundle $L_1 = \sum_{j \in S_1} D_j = \cO(3)$ is nef. Thus $B,S_1,\varnothing$ is a convex partition with basis for $\Theta$. That is, by scaffolding $f$ we obtain the cubic hypersurface as in Example~\ref{eg:cubic_forwards}.
\end{eg}

\begin{eg}[$dP_6$]
The projective plane blown up in three points, $dP_6$, is toric, but it has two famous models as a complete intersection:
\begin{enumerate}
\item as a hypersurface of type $(1,1,1)$ in $\PP^1\times\PP^1\times\PP^1$;
\item as the intersection of two bilinear equations in $\PP^2\times\PP^2$.
\end{enumerate}
Let us see how these arise from Laurent inversion. The Laurent polynomial mirror to $dP_6$ that we shall use is:
\[
f=x+y+\frac{1}{x}+\frac{1}{y}+\frac{x}{y}+\frac{y}{x}.
\]
We may scaffold $\Newt(f)$ in two different ways: using three triangles, and using a pair of squares:
\begin{center}
\includegraphics{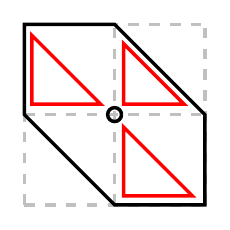}
\raisebox{30px}{$\quad\text{ and }\quad$}
\includegraphics{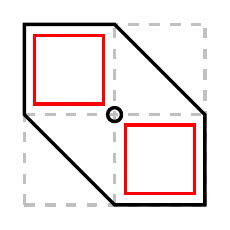}
\end{center}
These choices correspond, respectively, to the scaffoldings
\[
f=(1+x+y) +\frac{(1+x+y)}{x}+\frac{(1+x+y)}{y} - 3
\qquad\text{ and }\qquad
f=\frac{(1+x)(1+y)}{x}+\frac{(1+x)(1+y)}{y} - 2.
\]
As discussed, we ignore the constant terms.

From the first scaffolding we read off $r=3$, $k=1$, $B = \{1,2,3\}$, $S_1 = \{4,5,6\}$, $U = \varnothing$, and the exponents of the struts give:
\[
\cM = \left(
\begin{array}{ccc:ccc}
  1 & 0 & 0 & 1 & 0 & 0 \\
  0 & 1 & 0 & 0 & 1 & 0 \\
  0 & 0 & 1 & 0 & 0 & 1
\end{array}
\right)
\]
This gives GIT data $\Theta = (K;\LL;D_1,\ldots,D_6;\omega)$ with $K=(\Cstar)^3$, $\LL = \ZZ^3$, $D_1=D_4=(1,0,0)$, $D_2=D_5=(0,1,0)$, $D_3=D_6=(0,0,1)$, and $\omega=(1,1,1)$; the secondary fan here again has a unique maximal cone. The corresponding toric variety is $Y = \PP^1 \times \PP^1 \times \PP^1$. The line bundle $L_1 = \sum_{j \in S_1} D_j$ is $\cO(1,1,1)$, so we see that $f$ is a Laurent polynomial mirror to a hypersurface of type $(1,1,1)$ in $\PP^1 \times \PP^1 \times \PP^1$.

From the second scaffolding we read off $r=2$, $k=2$, $B=\{1,2\}$, $S_1 = \{3,4\}$, $S_2 = \{5,6\}$, $U = \varnothing$, and the exponents of the struts give:
\[
\cM = \left(
\begin{array}{cc:cccc}
1&0&0&1&1&0 \\
0&1&1&0&0&1 
\end{array}
\right)
\]
This gives GIT data $\Theta = (K;\LL;D_1,\ldots,D_6;\omega)$ with $K=(\Cstar)^2$, $\LL = \ZZ^2$, $D_1=D_4=D_5=(1,0)$, $D_2=D_3=D_6=(0,1)$, and $\omega=(1,1)$; once again the secondary fan has a unique maximal cone. The corresponding toric variety $Y$ is $\PP^2 \times \PP^2$. The line bundles $L_1 = D_3 + D_4$ and $L_2 = D_5 + D_6$ are both equal to $\cO(1,1)$, so we see that $f$ is a Laurent polynomial mirror to the complete intersection of two hypersurfaces defined by bilinear equations in $\PP^2 \times \PP^2$.

\end{eg}

\begin{eg} \label{eg:3-4}
  Consider the rigid maximally-mutable Laurent polynomial
  \begin{figure}
    \centering
    \includegraphics[scale=0.6]{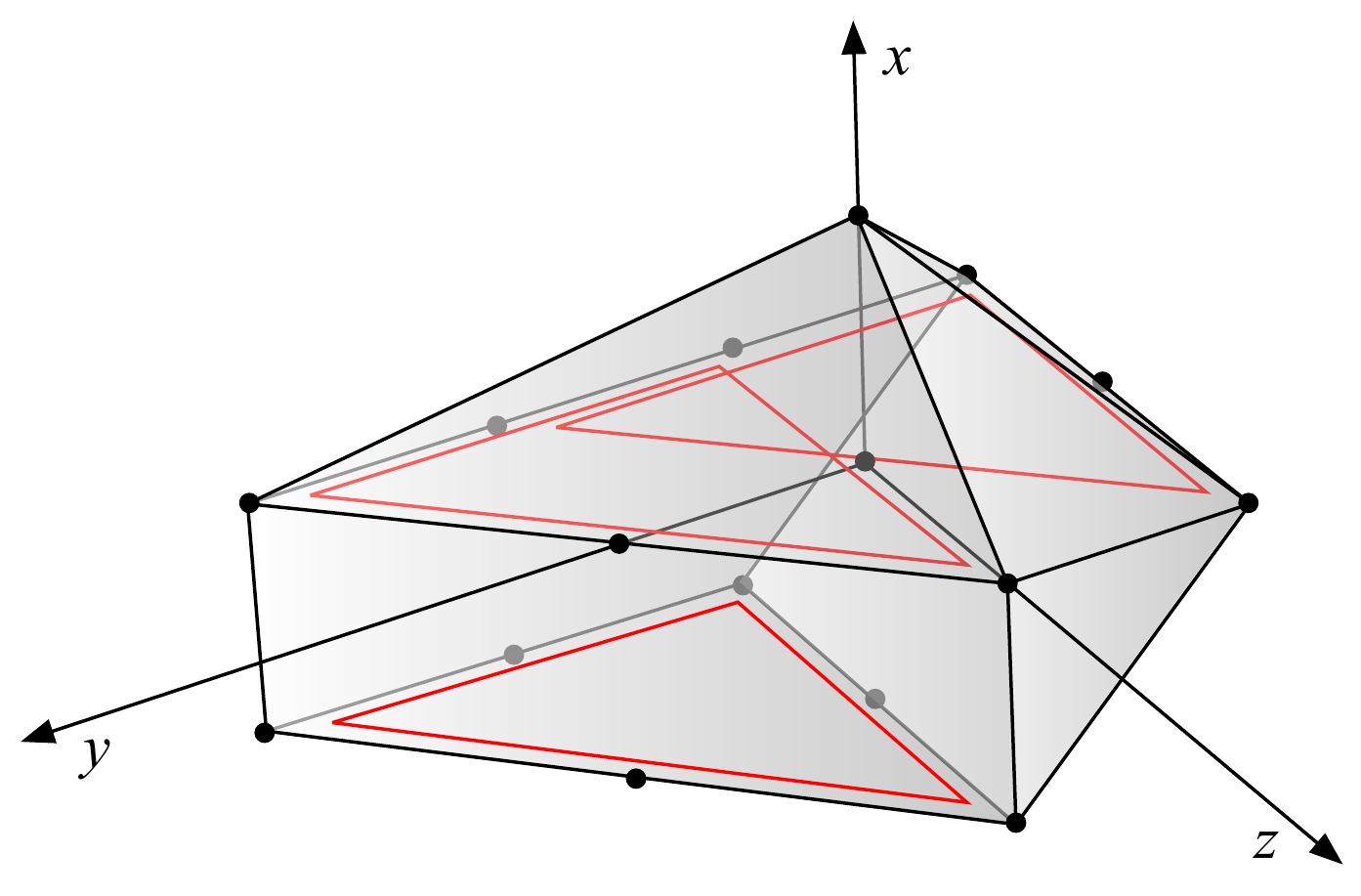}
    \caption{A scaffolding for $\Newt{f}$ in Example~\ref{eg:3-4}.}
    \label{fig:3-4}
  \end{figure}
  \[
  f = x + \frac{y^2}{z} + 2 y + \frac{3 y}{z} + z + \frac{3}{z} + \frac{z}{y} + \frac{2}{y} + \frac{1}{y z} + \frac{y^2}{x z} + \frac{2 y}{x} + \frac{2 y}{x z} + \frac{z}{x} + \frac{2}{x} + \frac{1}{x z}.
  \]
  The Newton polytope of $f$ can be scaffolded as in Figure~\ref{fig:3-4}, and there is a corresponding scaffolding of $f$:
  \[
  f = x + \frac{(1+y+z)^2}{xz} + \frac{(1+y+z)^2}{z} + \frac{(1+y+z)^2}{yz}
  \]
  From this we read off $r=3$, $k=1$, $B=\{1,2,3\}$, $U=\{4\}$, $S_1 = \{5,6,7\}$,  and the exponents of the struts give:
  \[
  \cM = \left(
    \begin{array}{ccc:cccc}
      1 & 0 & 0 & 1 & 1 & 0 & 1 \\
      0 & 1 & 0 & 0 & 1 & 0 & 1\\
      0 & 0 & 1 & 0 & 0 & 1 & 1
    \end{array}
  \right)
  \]
  This gives GIT data $\Theta = (K;\LL;D_1,\ldots,D_6;\omega)$ with $K=(\Cstar)^3$, $\LL = \ZZ^3$, $D_1=D_4=(1,0,0)$, $D_2=(0,1,0)$, $D_3=D_6=(0,0,1)$, $D_4=(1,1,0)$, and $D_7=(1,1,1)$.
  The secondary fan is as shown in Figure~\ref{fig:3-4_secondary_fan}. Choosing $\omega = (3,2,1)$ yields a weak Fano toric manifold $Y$ such that the line bundle $L_1 = \sum_{j \in S_1} D_j$ is convex. Let $X$ denote the hypersurface in $Y$ defined by a regular section of $L_1$. The class $-K_Y - L_1$ is nef but not ample on $Y$, but it becomes ample on restriction to $X$; thus $X$ is Fano (cf.~\cite[\S 57]{CCGK}). We see that $f$ is a Laurent polynomial mirror to $X$. This example shows that our Laurent inversion technique applies in cases where the ambient space $Y$ is not Fano. In fact $Y$ need not even be weak Fano.
  \begin{figure}
    \centering
    \begin{tikzpicture}[scale=0.5,arr/.style={thin,->,shorten >=2pt,shorten <=2pt,>=stealth}]
      \coordinate [label={below left:\tiny $(0,0,1)$}] (A) at (0, 0);
      \coordinate [label={above left:\tiny $(0,1,0)$}] (B) at (0, 6);
      \coordinate [label={below right:\tiny $(1,0,0)$}] (C) at (6,0);
      \coordinate [label={left:\tiny $(1,1,1)$}] (D) at (2,2);
      \coordinate [label={above right:\tiny $(1,1,0)$}] (E) at (3,3);
      
      \draw [thick] (A) -- (B) -- (E) -- (C) -- (A);
      \draw [thick] (A) -- (D) -- (E);
      \draw [thick] (B) -- (D) -- (C);
      
      \node (L) at (3.2,5.15) {\tiny $L_1$};
      \node (mK) at (5.55,2.2) {\tiny ${-K_Y}$};
      
      \coordinate (K) at (12/5,9/5);
      
      \draw [fill=black] (K) circle (0.1);
      \draw [fill=black] (2,2) circle (0.1);
      
      \draw[arr] (3,5) to [bend right=15] (2,2) ;
      \draw[arr] (5.3,2) to [bend left=5] (K) ;
    \end{tikzpicture}
    \caption{The secondary fan for Example~\ref{eg:3-4}, sliced by the plane $x+y+z=1$.}
    \label{fig:3-4_secondary_fan}
  \end{figure}
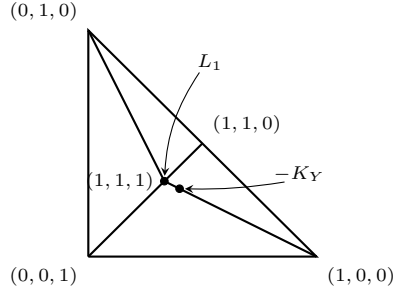
\end{eg}

\section{A New Four-Dimensional Fano Manifold} \label{sec:new_4d}

Consider 
\[
f = \frac{(1+x)^2}{xyw} + \frac{x}{z} + y + z + w
\]
This is a rigid maximally-mutable Laurent polynomial in four variables. It is presented in scaffolded form, and we read off
$r=2$, $k=1$, $B=\{1,2\}$, $S_1 = \{3,4\}$, $U=\{5,6,7\}$. The exponents of the struts give:
\[
\cM = \left(
  \begin{array}{cc:ccccc}
    1 & 0 & 1 & 1 & 1 & 0 & 1 \\
    0 & 1 & 1 & -1 & 0 & 1 & 0
  \end{array}
\right)
\]
This yields GIT data $\Theta = (K;\LL;D_1,\ldots,D_6;\omega)$ with $K=(\Cstar)^2$, $\LL = \ZZ^2$, $D_1=D_5=D_7=(1,0)$, $D_2=D_6=(0,1)$, $D_3=(1,1)$, and $D_4=(1,-1)$. We choose the stability condition $\omega = (5,2)$, thus obtaining a Fano toric orbifold $Y$ such that the line bundle $L_1 = D_3+D_4$ on $Y$ is convex.  Let $X$ denote the four-dimensional Fano manifold defined inside $Y$ by a regular section of $L_1$. 

The Fano manifold $X$ is new. To see this, we can compute the regularised quantum period $\widehat{G}_X$ of $X$. Since $f$ is a Laurent polynomial mirror to $X$, the regularised quantum period $\widehat{G}_X$ coincides with the classical period of $f$:
\begin{align*}
  \pi_f(t) = \sum_{d=0}^\infty c_d t^d && \text{where} && c_d = \coeff_1 \big(f^d\big).
\end{align*}
This is explained in detail in~\cite{CCGGK,CCGK}. In the case at hand, 
\[
\widehat{G}_X = \pi_f(t) = 1 + 12 t^3 + 120 t^5 + 540 t^6 + 20160 t^8 +  33600 t^9 + \cdots
\]
and we see that $\widehat{G}_X$ is not contained in the list of regularised quantum periods of known four-dimensional Fano manifolds~\cite{CKP,known_4d}. Thus $X$ is new. We did not find $X$ in our systematic search for four-dimensional Fano toric complete intersections~\cite{CKP}, because there we considered only ambient spaces that are Fano toric \emph{manifolds} whereas the ambient space $Y$ here has non-trivial orbifold structure. This is striking because the degree $K_X^4 = 433$ of $X$ is not that low -- compare with Figure~5 in~\cite{CKP}. In dimensions 2 and 3 only Fano manifolds of low degree fail to occur as complete intersections in toric manifolds.  The space $Y$ can be obtained as the unique non-trivial flip of the projective bundle $\PP\big(\cO(-1) \oplus \cO^{\oplus 3} \oplus \cO(1)\big)$ over $\PP^1$.  As was pointed out to us by Casagrande, the other extremal contraction of $X$, which is small, exhibits $X$ as the blow-up of $\PP^4$ in a plane conic.  This suggests that restricting to smooth ambient spaces when searching for Fano toric complete intersections may omit many Fano manifolds with simple classical constructions.


\section{From Laurent Inversion to Toric Degenerations}
\label{sec:embedded_degeneration}

Suppose now that we have a scaffolding~\eqref{eq:scaffolding} for the Laurent polynomial $f$, and that this gives rise to:
\begin{enumerate}
\item orbifold GIT data $\Theta = (K;\LL;D_1,\ldots,D_R;\omega)$;
\item a convex partition with basis $B, S_1, \ldots, S_k, U$ for $\Theta$; and
\item a choice of elements $s_i \in S_i$ for each $i \in [k]$.
\end{enumerate}
We now explain how to pass from this data to a toric degeneration of the complete intersection $X \subset Y$ defined by a regular section of the vector bundle $\oplus_i L_i$. This degeneration was discovered independently by Doran--Harder~\cite{Doran--Harder}; see~\S\ref{sec:torus_charts} for an alternative view on their construction. In favourable circumstances, as we will explain, the central fiber of this toric degeneration is the Fano toric variety $X_f$ defined by the spanning fan of $\Newt{f}$. The existence of such a degeneration is predicted by Mirror Symmetry.

By assumption we have, as in~\S\ref{sec:forward}, an $r \times R$ matrix $\cM = (m_{i,j})$ of the form:
\[
\cM = \left(
\begin{array}{c:ccc}
  \multirow{3}{*}{$\quad I_r\quad$} & m_{1,r+1} & \cdots & m_{1,R} \\
  & \vdots & & \vdots \\
  & m_{r,r+1} & \cdots & m_{r,R} \\
\end{array}
\right)
\]
such that $l_{b,i} := \sum_{j \in S_i} m_{b,j}$ is non-negative for all $b \in [r]$ and $i \in [k]$. The exact sequence~\eqref{eq:fan_sequence} becomes
\[
\xymatrix{
  0 \ar[r] & \ZZ^r \ar[r]^{\cM^T} & \ZZ^R \ar[r]^\rho  & N \ar[r] & 0
}
\]
and, writing $\rho_i \in N$ for the image under $\rho$ of the $i$th standard basis vector in $\ZZ^R$, we find that $\{ \rho_i\mid r<i\leq R\}$ is a distinguished basis for $N$ and that
\begin{align*}
  \rho_i = -\sum_{j=r+1}^R m_{i,j} \rho_j && \text{for all $i \in [r]$.}
\end{align*}
Let $M = \Hom(N,\ZZ)$ and define $u_j\in M$, $j\in [k]$, by
\[
u_j(\rho_i)=\begin{cases}
0&\text{if $r < i \leq R$ and $i\not\in S_j$};\\
1&\text{if $r < i \leq R$ and $i\in S_j$}.\\
\end{cases}
\]
Let $N':=N\cap H_{u_1}\cap\ldots\cap H_{u_k}$ be the sublattice of $N$ given by restricting to the intersection of the hyperplanes $H_{u_i}:=\{v\in N\mid u_i(v)=0\}$. Let $\Sigma'$ denote the fan defined by intersecting $\Sigma$ with $N'_\QQ$, and let $X'$ be the toric variety defined by $\Sigma'$.

\begin{pro}
  There is a flat degeneration $\XX \to \AA^1$ with general fiber $\XX_t$ isomorphic to $X$ and special fiber $\XX_0$ isomorphic to $X'$.
\end{pro}

\begin{proof}
  Recall that $X$ is cut out of the toric variety $Y$ by regular sections $s_i$ of the line bundles $L_i$, $i \in [k]$. By deforming $s_i$ to the binomial section $s_i'$ of $L_i$ given by
  \[
  s_i = \prod_{a \in [r]} x_a^{l_{a,i}} - \prod_{j \in S_i} x_j
  \]
  we can construct a flat degeneration with general fiber $X$ and special fiber a toric variety $X'$. Since $u_i(\rho_a) = {- l_{a,i}}$, we see that the fan $\Sigma'$ defining $X'$ is the intersection of the fan $\Sigma$ defining $Y$ with $H_{u_1} \cap \cdots \cap H_{u_k}$, as claimed.
\end{proof}

Our choice of elements $s_i \in S_i$, $i \in [k]$, gives rise to a distinguished basis for $N'$, consisting of
\begin{equation}
  \label{eq:which_basis}
  \begin{aligned}
    \rho_i,\quad&\text{where $i \in U$},\\
    \text{ and }\quad
    \rho_i - \rho_{s_j},\quad&\text{where $i \in S_j \setminus \{s_j\}$ for some $j \in [k]$.}
  \end{aligned}
\end{equation}
Comparing~\eqref{eq:which_variables} with~\eqref{eq:which_basis}, we see that this choice of basis also specifies an isomorphism between $N'$ and the lattice $A$ that contains $\Newt{f}$. Thus it makes sense to ask whether the fan $\Sigma'$ coincides with the spanning fan of $\Newt{f}$; in this case we will say that $\Sigma'$ \emph{is the spanning fan}. If $\Sigma'$ is the spanning fan then the above construction gives a degeneration from $X$ to the (singular) toric variety $X_f$, as predicted by Mirror Symmetry.

\begin{rem}
  In any given example it is easy to check whether $\Sigma'$ is the spanning fan. This is often the case -- it holds, for example, for all of the examples in this paper -- but it is certainly not the case in general. It would be interesting to find a geometrically meaningful condition that guarantees that $\Sigma'$ is the spanning fan. This problem is challenging because, at this level of generality, we do not have much control over what the fan $\Sigma$ looks like. It is easy to see that each ray of $\Sigma'$ passes through some vertex of a strut in the scaffolding of $\Newt{f}$, and that the cone $C_a' \subset N'_\QQ$ over the strut $\Delta_a = \Newt{f_a}$ is given by the intersection with $N'_\QQ$ of the cone $C_a \subset N$ spanned by $\{ \rho_a \} \cup \{\rho_i\mid i \in S_1 \cup \cdots \cup S_k\}$. But typically only some of the $C_a$ lie in $\Sigma$ (indeed typically the cones $C_a'$ overlap with each other) and in general it is hard to say more. Doran--Harder~\cite{Doran--Harder} give sufficient conditions for $\Sigma'$ to be a refinement of the spanning fan, but for applications to Mirror Symmetry this is not enough.
\end{rem}

\section{Torus Charts on Landau--Ginzburg Models}
\label{sec:torus_charts}

Suppose, as before, that we have:
\begin{equation}
  \label{eq:LG_data}
  \begin{minipage}{0.93\linewidth}
    \begin{enumerate}
    \item orbifold GIT data $\Theta = (K;\LL;D_1,\ldots,D_R;\omega)$;
    \item a convex partition with basis $B, S_1, \ldots, S_k, U$ for $\Theta$; and
    \item a choice of elements $s_i \in S_i$ for each $i \in [k]$.
    \end{enumerate}
  \end{minipage}
\end{equation}
Let $Y$ be the corresponding toric orbifold, and $X \subset Y$ the complete intersection defined by a regular section of the vector bundle $\oplus_i L_i$. Givental~\cite{Givental:toric} and Hori--Vafa~\cite{Hori--Vafa} have defined a Landau--Ginzburg model that corresponds to $X$ under Mirror Symmetry. In this section we explain how to write down a torus chart on the Givental/Hori--Vafa mirror model on which the superpotential restricts to a Laurent polynomial. This gives an alternative perspective on Doran--Harder's notion of \emph{amenable collection subordinate to a nef partition}~\cite[\S\S2.2--2.3]{Doran--Harder}.

\begin{dfn} 
  Suppose that we have fixed orbifold GIT data $\Theta$ defining $Y$, as in~(\ref{eq:LG_data}-i). The Landau--Ginzburg model mirror to $Y$ is the family of tori equipped with a superpotential:
  \[
  \xymatrix{ (\Cstar)^R \ar^-{W}[rr] \ar_D[d] & & \CC \\
    T_{\mathbb{L}^*} & & }
  \]
  where $W= \sum^R_{j = 1}{x_j}$;  $x_1,\ldots,x_R$ are the standard co-ordinates on $(\Cstar)^R$; $D$ is the map from~\eqref{eq:divisor_sequence}; and $T_{\LL^*}$ is the torus $\LL^* \otimes \Cstar$.
\end{dfn}

\noindent In our context, rather than considering the whole family over $T_{\mathbb{L}^*}$, we restrict to the fiber over~$1$. Extending the diagram defining the Landau--Ginzburg model to include this fiber we have:
\[
    \xymatrix{T_M \ar^-{\rho^\vee}[rr] && (\Cstar)^R \ar^-{W}[rr] \ar_D[d] & & \CC \\
     && T_{\mathbb{L}^*} & & }
\]
where $T_M = M \otimes \Cstar$ and $\rho^\vee$ is the dual to the fan map $\rho$ from~\eqref{eq:fan_sequence}.

\begin{dfn}
  Suppose that we have fixed orbifold GIT data and a nef partition with basis, as in~\eqref{eq:LG_data}. The Landau--Ginzburg model mirror to $X$ is the restriction of the mirror model for $Y$ to a subvariety $X^\vee$, defined by the following commutative diagram:
\[
    \xymatrix{
     & & \CC & & \\
    X^\vee \ar^-j[rr] \ar[urr] & & T_M \ar^-{\rho^\vee}[rr] \ar[d] & & (\Cstar)^R \ar_-{W}[ull] \ar^D[d] \ar^{\Phi}[dll]\\
    & & \CC^l & & T_{\mathbb{L}^*} }
\]
where $\Phi := \big( \sum_{i \in S_1}{x_i},\ldots, \sum_{i \in S_k}{x_i} \big)$ and $j$ is the inclusion of the fiber over $1$. The Landau--Ginzburg model mirror to $X$ is the map
\[
\big(\rho^\vee \circ j\big)^* W \colon X^\vee \rightarrow \CC
\]
\end{dfn}

We now present a general technique for finding torus charts on $X^\vee$ on which the restriction of the superpotential $\big(\rho^\vee \circ j\big)^* W$ is a Laurent polynomial. To do this we will construct a birational map $\mu$ such that the pullback $\chi := (\rho^\vee \circ \mu)^*\Phi$ of $\Phi$ becomes regular, as in the following diagram.
\[
\xymatrix{
& & X^\vee \ar^j[rr] & & T_M \ar^{(\rho^\vee)^*\Phi}[d] \\
T_{\text{Ker}(\chi)} \ar_{\theta := \text{ker}(\chi)}[rr] \ar@{-->}[urr] & & T_M \ar@{-->}[urr]^\mu \ar[rr]_{\chi} & & \CC^l
}
\]

\begin{rem}\label{rem:notation}
Via the bijection between monomials in the variables $x_i$, $1 \leq i \leq R$, with their exponents in $\ZZ^R$ we identify the monomials $(\rho^\vee)^*(x_i)$ with their exponents $\rho_i \in N$.
In this notation:
\[
(\rho^\vee)^*\Phi = \Big(\sum_{i \in S_1}{x^{\rho_i}} , \cdots , \sum_{i \in S_k}{x^{\rho_i}} \Big)
\] 
Recall that the vectors $\rho_i$ generate the rays of the fan $\Sigma$ that defines $Y$.
\end{rem}

We construct our birational map $\mu$ from the data in~\eqref{eq:LG_data} together with a choice of lattice vectors $w_i \in M$ such that:
\begin{equation}
  \label{eq:weight_vectors}
  \begin{minipage}{0.93\linewidth}
    \begin{enumerate}
    \item $\langle w_i , \rho_j \rangle = -1$ for all $j \in S_i$ and all $i$;
    \item $\langle w_i , \rho_j \rangle = 0$ for all $j \in S_l$ such that $l < i$ and all $i$;
    \item $\langle w_i, \rho_j \rangle \geq 0$ for all $j \in S_l$ such that $l > i$ and all $i$.
    \end{enumerate}
  \end{minipage}
\end{equation}
This is exactly Doran--Harder's notion of an \emph{amenable collection subordinate to a nef partition}. 

\begin{dfn}
  A \emph{weight vector} $w \in M$ and a \emph{factor} $F \in \CC[w^\perp]$ together determine a birational transformation $\theta \colon T_M \dashrightarrow T_M$ called an \emph{algebraic mutation}. This is given by the automorphism $x^\gamma \mapsto x^\gamma F^{\langle \gamma,w\rangle}$ of the field of fractions $\CC(N)$ of $\CC[N]$.
\end{dfn}

\noindent We define the birational map $\mu$ as the composition of a sequence of algebraic mutations $\mu_1, \ldots, \mu_k$, where the mutation $\mu_i$ has weight vector $w_i$ and factor given by 
\[
F_i := \frac{\big(\mu_1 \circ \cdots \circ \mu_{i-1}\big)^*\bigg(\sum_{j \in S_i}{x^{\rho_j}}\bigg)}{x^{\rho_{s_i}}}
\]
The conditions~\eqref{eq:weight_vectors} guarantee that $F_i$ is a Laurent polynomial, that $F_i \in \CC[w_i^\perp]$, and that $\big(\mu_1 \circ \cdots \circ \mu_i\big)^* W$ is a Laurent polynomial for all $i \in [k]$.

We can always take the weight vectors $w_i$ in~\eqref{eq:weight_vectors} to be equal to the $-u_i$ from~\S\ref{sec:embedded_degeneration}, but many other choices are possible. We get a toric degeneration in this more general context, too (cf.~\cite{Doran--Harder}):

\begin{lem}
  The lattice vector $w_i \in M$ defines a binomial section of the line bundle $L_i \in \Pic(Y)$.
\end{lem}
\begin{proof}
The lattice $M$ is the character lattice of the torus $T_N$, and so $w_i$ defines a rational function on $Y$. The image $\rho^\vee(w_i) \in (\ZZ^*)^R$ defines a pair of effective torus invariant divisors by taking the positive entries and minus the negative entries of this vector, written in the standard basis. The only negative entries are those in $S_i$, which are equal to minus one. Both monomials have the same image under $D$, and so they are both in the linear system defined by $L_i$.
\end{proof}

\section*{Acknowledgements}
We thank Alessio Corti and Andrea Petracci for many useful conversations, and Cinzia Casagrande for an extremely helpful observation about the Fano manifold in \S\ref{sec:new_4d}. TC and AK were supported by ERC Starting Investigator Grant 240123. TP was supported by an EPSRC Prize Studentship. TC thanks the University of California at Berkeley for hospitality during the writing of this paper.

\bibliographystyle{plain}
\bibliography{bibliography}
\end{document}